\documentclass[12pt,reqno]{amsart}

\usepackage{amsthm,amsmath,amssymb}
\usepackage{graphicx,color}
\usepackage[colorlinks=true]{hyperref}
\hypersetup{urlcolor=blue, citecolor=red}

\usepackage{enumitem}

\newtheorem{theorem}{Theorem}[section]

\newtheorem{corollary}[theorem]{Corollary}

\newtheorem{lemma}[theorem]{Lemma}

\newtheorem{remark}[theorem]{Remark}

\def\R{\mathbb{R}}
\def\eps{\varepsilon}

\title{Generalized elastica problems under area constraint}
\author[V. Ferone, B. Kawohl, C. Nitsch]
       {Vincenzo Ferone, 
       Bernd Kawohl, 
       Carlo Nitsch
       }
\address[V.~Ferone]{Universit\`a degli Studi di Napoli Federico II (Italy).}
\email{ferone@unina.it}
\address[B.~Kawohl]{Universit\"at zu K\"oln (Germany).}
\email{kawohl@math.uni-koeln.de}
\address[C.~Nitsch]{Universit\`a degli Studi di Napoli Federico II (Italy).}
\email{c.nitsch@unina.it}

\date{\today}

\begin{document}


\begin{abstract}
It was recently proved in \cite{BH,FKN} that the elastic energy $E(\gamma)=\tfrac{1}{2}\int_\gamma\kappa^2 ds$ of a closed curve $\gamma$ with curvature $\kappa$ has a minimizer among all plane, simple, regular and closed curves of given enclosed area $A(\gamma)$, and that the minimum is attained only for circles. In particular, the proof used in \cite{FKN} is of a geometric nature, and here we show under which hypothesis it can be extended to other functionals involving the curvature. As an example we show that the optimal shape remains a circle for the $p$-elastic energy $\int_\gamma|\kappa|^p ds$, whenever $p>1$. 
\end{abstract}

\maketitle

\section{introduction}
It was proved in \cite{BH,FKN} that the elastic energy 
$$E(\gamma):=\frac{1}{2}\int_\gamma \kappa^2$$
of a simple closed curve $\gamma$ with curvature $\kappa$, enclosing a given amount of area, is minimal when the curve is a circle. An equivalent way to formulate the result is to say that for all $\gamma$ smooth, simple and closed
\begin{equation}\label{linear}
E^2(\gamma){A}(\gamma)\ge \pi^3,
\end{equation}
with equality only if $\gamma$ is a circle. Here ${A(\gamma)}$ denotes the area enclosed by $\gamma$.
In this paper we study the non linear case with elastic energy of the form
\[
E_f(\gamma):=\int_\gamma f(\kappa(s))\, ds,
\]
and we look for conditions on $f$ under which the circle still minimizes $E_f(\gamma)$ among simply connected sets of given area.\\
Even though the reduction to the convex case is in general not straightforward, the study of convex curves, i.e. curves bounding convex domains, is a fundamental building block in the proofs of \eqref{linear}, see \cite{BH,FKN}. For convex curves \eqref{linear} has been known since a work of Gage \cite{Gage1983}. Our first result is indeed 
\begin{theorem}\label{prop}
If $p>1$ and $f(t)=t^p$ for $t\ge0$
then the disc, and only the disc, minimizes $E_f(\partial \Omega)$ among convex sets $\Omega$ of given area and boundary of class $W^{2,p}$.
\end{theorem}
Then we prove sufficient conditions which allow us to extend the minimality of the circle from the class of convex curves (or domains) to nonconvex ones.
\begin{theorem}\label{main}
If a measurable function $f:\R\rightarrow[0,+\infty]$ satisfies
\begin{enumerate}
\item $f(0)=0$ and $f(s)>0$ when $s>0$,
\item $g(t)=f(\tfrac1{\sqrt t})\sqrt t$ is convex for $t>0$,
\end{enumerate}
and if the disc minimizes $E_f(\partial \Omega)$ among smooth convex sets $\Omega$ of given area, then it also minimizes $E_f(\partial \Omega)$ in the larger class of smooth simply connected sets $\Omega$ of given area.  Moreover, if the disc is the unique minimizer among convex sets, the same is true also among simply connected sets.
\end{theorem}
Here by smooth curve we mean $C^1$, piecewise $C^2$ and with curvature which changes sign a finite number of times. Analytic curves serve as an example. 
\begin{remark} The assumptions (1) and (2) on $f$ are, in particular, satisfied by all nonnegative functions, positive in $(0,\infty]$, convex in $[0,\infty]$, which vanish in zero. Two remarkable examples are given by $f(t)=|t|^p$ and $f(t)=t_+^p:=\left(\max\{t,0\}\right)^p$. Also notice that (1) and (2) do not require the convexity of $f$. 
\end{remark}
Hence from Theorem \ref{main} and Theorem \ref{prop} we deduce the following nonlinear generalizations of \eqref{linear}.
\begin{corollary}\label{corollary}
For any $p\in(1,\infty]$ and for the energies
\[F_p(\gamma):=\left(\frac{1}{2}\right)^{\frac2p}\|\kappa\|_p^2\qquad\hbox{or}\qquad
\widetilde F_p(\gamma):=\left(\frac{1}{2}\right)^{\frac2p}\|\kappa_+\|_p^2
\]
the following inequalities hold true
\begin{equation}\label{isop}
F_p(\gamma)^{\frac p{p-1}}\,A(\gamma)\ge \pi^{\frac{p+1}{p-1}}\qquad\hbox{or}\qquad
\widetilde F_p(\gamma)^{\frac p{p-1}}\,A(\gamma)\ge \pi^{\frac{p+1}{p-1}},
\end{equation}
for all curves in $W^{2,p}$, with equality only for the disc.
\end{corollary}
The structure of the paper is as follows. In Section \ref{sec_prop} we prove Theorem \ref{prop}. We first show that a sequence of convex sets with area and energy both equibounded, has also equibounded diameter. Hence we prove compactness of a minimizing sequence of domains. Then we show that a minimum is necessarily uniformly convex in the sense that the curvature, as an $L^p$ function, is essentially bounded away from zero. In view of the uniform convexity one can then perform a domain variation and work out the Euler Lagrange equation which the boundary of an energy minimizer has to satisfy. It turns out, using a result by Ben Andrews \cite{A}, that the disc is the unique stationary point and therefore also the unique energy minimizer. \\
In Section \ref{sec_theorem} we prove Theorem \ref{main} and Corollary \ref{corollary}. The proof of Theorem \ref{main} is mainly contained in \cite{FKN}, where the case $p=2$ is studied in details. Here we only need to understand the proof in a more general framework.\\
Throughout the paper $(\bf{e_1},{\bf e_2})$ will be the canonical orthonormal basis of $\mathbb{R}^2$. For a smooth regular planar curve $\gamma:[0,L]\to \mathbb{R}^2$ parametrized by arc length $s$, by convention we define the normal vector ${\bf n}(s)$ so that ${\bf n}(s)$ and  the unit tangent vector ${\bf t}(s)=\gamma'(s)$ form for all $s\in [0,L]$ a basis $({\bf t},{\bf n})$ that has the same orientation as the basis $(\bf{e_1},{\bf e_2})$. The scalar signed curvature $\kappa(s)$ is defined by
$$\frac{d}{ds}{\bf t}(s)=:\kappa(s)\,{\bf n}(s).$$
In particular for $\gamma\in W^{2,p}([0,L]; \mathbb{R}^2)$, 
 the norm $\|\kappa\|_p:=\left(\int_0^L\kappa(s)ds\right)^{1/p}$ is well defined.
\section{Proof of Theorem \ref{prop}}\label{sec_prop}
To prove Theorem \ref{prop} we need two preliminary Lemmata.
We parametrize a simple closed curve $\gamma$ by arc length $s$, with $(x(s),y(s)):[0,L]\rightarrow\R^2$ running counterclockwise, and we denote by $\theta(s):[0,L]\rightarrow[0,2\pi]$ the angle between the tangent ${\bf t}(s)$ and $\bf e_1$. Then the parametrization is
\[
x(s)=x(0)+\int_0^s\cos\theta(t) dt, \qquad y(s)=y(0)+\int_0^s\sin\theta(t) dt.
\]
This choice implies that $\theta$ is monotone nondecreasing if $\gamma$ is the boundary of a convex set, and that $\theta$ is a.e. differentiable with $\theta'=\kappa$. Within this section we are considering the case where $\kappa$ is nonnegative and $f(\kappa)=\kappa^p$. Minimization of $E_f$ is analogous to minimization of \[F_p(\gamma):=\left(\frac{1}{2}\right)^{\frac2p}\|\kappa\|_p^2.\] 
 If not otherwise specified we always assume that the set $\Omega$ is oriented so that $$\theta(0)=0.$$ 
The first result is

\begin{lemma}\label{existence}
For any given $a>0$ there exists a minimizer of $F_p(\partial \Omega)$ among convex sets $\Omega$ of area $a$.
\end{lemma}
\begin{proof}
The proof follows from the Blaschke selection theorem and the lower semicontinuity of energy once we establish the boundedness of any minimizing sequence. 

We parametrize the boundary of a convex set $\Omega$ of area $a$ as above. Therefore $\theta$ is a nondecreasing function and we have:
\begin{equation}\label{ineqtheta}
|\theta(s)|\le \int_0^s|\theta'(s)|\,ds\le 2^{\frac1p}s^{\frac{p-1}{p}}F_p(\gamma)^{\frac12}.
\end{equation}
For $s=L$ this turns into
\begin{equation}\label{ineqel}
L\ge2 \left(\frac{\pi}{F_p(\gamma)^{\frac12}}\right)^{\frac{p}{p-1}}.
\end{equation}
Choose $s_0=\left(\frac{\pi}{2^{\frac{2p+1}{p}}F_p(\gamma)^{\frac12}}\right)^{\frac{p}{p-1}}$ and notice that by (\ref{ineqel}) we have $s_0<L$, while by (\ref{ineqtheta}) we have $\theta(s)\le \frac\pi4$ for all $0\le s\le s_0$.
Consequently
\begin{equation*}
x(s_0)-x(0)=\int_0^{s_0}\cos\theta(s)\,ds\ge s_0 \cos\theta(s_0)\ge \frac{s_0}{\sqrt2}
\end{equation*}
and analogously
\begin{equation*}
x(L-s_0)-x(0)\le -\frac{s_0}{\sqrt2}.
\end{equation*}
Thus, the width $w$ of $\Omega$ satisfies $w(\Omega)\ge \sqrt2 s_0$.
Since by an old result of Kubota, see \cite{Ku} or \cite{SA}, the diameter $d$ of a convex planar set $\Omega$ is bounded from above in terms of area and width by  $d\le \frac{2a}{w(\Omega)}$, we may conclude that 
\begin{equation}\label{diam}
d\le 2^{\frac{5p+1}{2(p-1)}}a\left(\frac{F_p(\gamma)^{\frac12}}\pi\right)^{\frac{p}{p-1}}.
\end{equation}
Therefore $d$ is bounded in terms of the energy and area alone.
\end{proof}

\begin{lemma}\label{centrosymm}
For any given $a>0$ there exists a centrosymmetric minimizer of $F_p(\partial \Omega)$ among convex sets $\Omega$ of area $a$. Moreover, the curvature of the boundary of such a minimizer is bounded away from 0. 
\end{lemma}
\begin{proof} Since the Lemma contains two claims, we prove them one after the other.
\vskip.2cm

\noindent\sc Step 1. \rm Existence of a centrosymmetric minimizer.
\vskip.1cm

\noindent Let $\Omega$ be a minimizer. Then there exist two points on $\partial \Omega$ with parallel tangents such that the chord connecting these points splits $\Omega$ into two parts $\Omega_1$ and $\Omega_2$ of equal area. Let $\Omega_1$ be such that $F_p(\partial\Omega_1\cap\partial \Omega)\le F_p(\partial\Omega_2\cap\partial \Omega)$. Then the set $\widetilde\Omega$ obtained as the union of $\Omega_1$ and of its rotation by $\pi$ around the midpoint of the chord is centrosymmetric, has the same area as $\Omega$ and $\partial\widetilde\Omega$ has at most the same nonlinear elastic energy $F_p$ as $\partial \Omega$.
\vskip.2cm

\noindent\sc Step 2. \rm Any centrosymmetric minimizer has curvature which is bounded away from zero.
\vskip.1cm

\noindent The usual parametrization of the boundary of a centrosymmetic convex set $\Omega$ is characterized by the properties
\begin{itemize}
\item $\theta$ is monotone nondecreasing,
\item $\theta(s+\tfrac L2)=\pi+\theta(s)$, for $s\in[0, L/2]$.
\end{itemize}

We want to show that, for a.e.\ $s\in[0, L]$,
\begin{equation}\label{deltabounda}
\theta'(s)\ge\sqrt{F_p(\partial \Omega)}\left((p-1)\frac{w(\Omega)}{a2^{p+1}}\right)^{\frac1p}.
\end{equation}
Without loss of generality we can suppose that $\theta(s)$ is differentiable at $s=0$ and $\theta'(0)=\delta$, and prove that
\begin{equation}\label{deltabound}
\delta\ge\sqrt{F_p(\partial \Omega)}\left((p-1)\frac{w(\Omega)}{a2^{p+1}}\right)^{\frac1p}.
\end{equation}
To this aim we start by building a centrosymmetric set $\Omega_\eps$ for small $\eps>0$,  which is characterized by $\theta_\eps(s)$ 
\begin{equation}\label{thetaeps}
\theta_\eps(s):=
\left\{
\begin{array}{ll}
2s\theta(\eps)
 &   s\in[0,\eps/2]  \\
\theta(\eps)   &   s\in[\eps/2,\eps] \\
\theta(s) &  s\in[\eps,L/2],
\end{array}
\right.
\end{equation}
and which still satisfies $\theta_\eps(s+\tfrac L2)=\pi+\theta_\eps(s)$, for $s\in[0, L/2]$.
Using the above function in the definition of elastic energy, we get:
\begin{equation}\label{deltaE}
F_p(\partial\Omega_\eps)^\frac p2\le F_p(\partial\Omega)^\frac p2+\eps\delta^p2^{p-1}+o(\eps),
\quad\text{as}\>\eps\rightarrow0.
\end{equation}
By the Gauss-Green formula we have
\[a=area(\Omega)=\frac12\int_0^L\int_0^s\sin(\theta(s)-\theta(t))\,dt\,ds
\]
and
\[a_\eps=area(\Omega_\eps)=\frac12\int_0^L\int_0^s\sin(\theta_\eps(s)-\theta_\eps(t))\,dt\,ds.
\]
Thus
\begin{eqnarray}\label{deltaA}
|a_\eps- a|\le \frac12\int_0^L\int_0^s|\theta_\eps(s)-\theta(s)|\,dt\,ds+\frac12\int_0^L\int_0^s|\theta_\eps(t)-\theta(t)|\,dt\,ds\le L\eps\theta(\eps)=o(\eps).
\end{eqnarray}

We observe that, since $\theta_\eps(s)$ is constant in $[\eps/2,\eps]\cup[\tfrac L2+\eps/2,\tfrac L2+\eps]$, the boundary of $\Omega_\eps$ contains two parallel segments of length $\eps/2$.
Then we can modify again $\Omega_\eps$ by removing the shaded area as described in Figure \ref{omega_hat}. 
\begin{figure}
\begin{center}
\def\svgwidth{\textwidth}
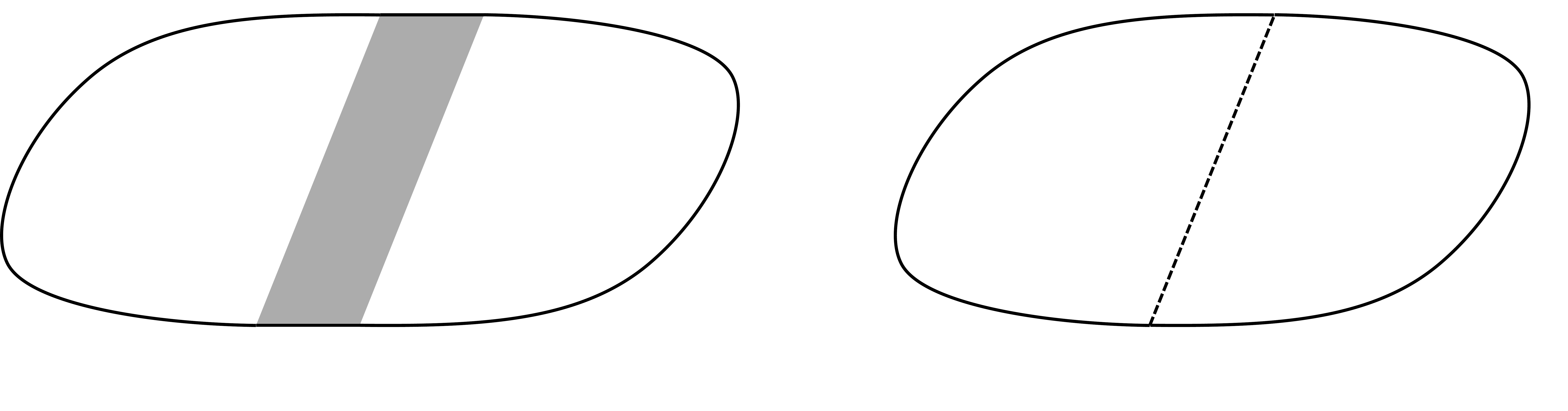
\end{center}
\caption{Construction of $\widehat\Omega_\eps$ from $\Omega_\eps$}\label{omega_hat}
\end{figure}
Thus we obtain a new set $\widehat\Omega_\eps$ of smaller area $\hat a_\eps=area(\widehat\Omega_\eps)$ and same elastic energy as $\Omega_\eps$. In fact,
\begin{equation}\label{deltaAA}
a_\eps-\hat a_\eps
\ge \frac{w(\Omega_\eps)}2\eps=\frac{w(\Omega)}2\eps+o(\eps).
\end{equation}
Collecting \eqref{deltaE}, \eqref{deltaA} and \eqref{deltaAA} we have:
\begin{flalign}
F_p(\partial\Omega)^{\frac p{p-1}}\,a-&F_p(\partial\widehat\Omega_\eps)^{\frac p{p-1}}\,\hat a_\eps\\\notag
&\ge a\, F_p(\partial\Omega)^{\frac p{p-1}}-\left(a+\eps \frac{w(\Omega)}{2}\right)\left(F_p(\partial\Omega)^{\frac p2}+\eps\delta^p2^{p-1}\right)^\frac 2{p-1}+o(\eps)\\\notag
&=\eps\,a\, F_p(\partial\Omega)^{\frac p{p-1}}\left(\frac{w(\Omega)}{2a}-\frac{2^p\delta^p}{(p-1)F_p(\partial\Omega)^{\frac p2}}\right)+o(\eps).
\end{flalign}
Minimality of $\Omega$ implies that the right hand side cannot be positive, that is \eqref{deltabound}.
\end{proof}

\begin{proof}[Proof of Theorem \ref{prop}]
In view of Lemma \ref{centrosymm} we can now compute the Euler equation for the minimizing domain $\Omega$.
Let us consider a set $\Omega_\eps$ such that $\partial\Omega_\eps$ is described by $\theta_\eps(s)=\theta(s)+\eps\psi(s)$ and observe that in view of the lower bound on $\theta'(s)$, the function $\theta_\eps(s)$ characterizes again a centrosymmetric convex set as long as $\psi(s)$ is a smooth function, periodic with period $L/2$, and $\eps>0$ is small enough.
The variation of the functional $F_p(\partial\Omega_\eps)^{\frac p{p-1}}area(\Omega_\eps)$ reads as follows:
\begin{flalign}
\frac d{d\eps}&F_p(\partial\Omega_\eps)^{\frac p{p-1}}area(\Omega_\eps)\biggl|_{\eps=0}=\notag\\
&=
\left(\frac12\right)^{\frac2{p-1}}\left(\int_0^L \theta'(s)^p\,ds\right)^{\frac{3-p}{p-1}}\biggl[
\frac{2p}{p-1}area(\Omega)\int_0^L\theta'(s)^{p-1}\psi'(s)\,ds+\notag\\
&+\frac12\left(\int_0^L \theta'(s)^p\,ds\right)\int_0^L\int_0^s\cos(\theta(s)-\theta(t))(\psi(s)-\psi(t))\,dt\,ds
\biggr].
\end{flalign}
Since by optimality the first variation vanishes, we get
\begin{flalign}
\frac{4p}{p-1}&area(\Omega)\int_0^L\theta'(s)^{p-1}\psi'(s)\,ds+\notag\\
&+\left(\int_0^L \theta'(s)^p\,ds\right)\int_0^L\int_0^s\cos(\theta(s)-\theta(t))(\psi(s)-\psi(t))\,dt\,ds=0.
\end{flalign}
Developing the last integral we obtain
\begin{flalign}\label{ie}
\frac{2p}{p-1}&area(\Omega)\int_0^L\theta'(s)^{p-1}\psi'(s)\,ds+\notag\\
&+\left(\int_0^L \theta'(s)^p\,ds\right)\int_0^L(x(s)-x(0),y(s)-y(0))(x'(s),y'(s))\psi(s)\,ds=0
\end{flalign}
Here we have used the fact that
\begin{flalign}
\int_0^L\int_0^s\cos(\theta(s)-\theta(t))\psi(t)\,dt\,ds=\int_0^L\int_t^L\cos(\theta(s)-\theta(t))\psi(t)\,ds\,dt\notag\\
=-\int_0^L\int_0^t\cos(\theta(s)-\theta(t))\psi(t)\,ds\,dt.\notag
\end{flalign}
If $h(s)=(x(s)-x(0),y(s)-y(0))((y'(s),-x'(s))$, we have $h'(s)=\theta'(s)(x(s)-x(0),y(s)-y(0))((x'(s),y'(s))$, then, by arbitrariness of $\psi$ among periodic functions, equation (\ref{ie}) becomes
\begin{equation}\label{ha}
\frac{2area(\Omega)}{\int_0^L \theta'(s)^p\,ds}(\theta'(s)^{p}+\theta'(s+\tfrac L2)^{p})'=({h(s)+h(s+\tfrac L2)})',\qquad \text{in }(0,\tfrac L2).
\end{equation}
We recall that $\theta'(s)$ is periodic and we observe that $(h(s)+h(s+\tfrac L2))=2(\gamma(s)-\bar\gamma,{\bf n})$, where $\bar \gamma$ is the center of symmetry of $\Omega$.
So \eqref{ha} can be written as
\begin{equation}\label{hb}
\frac{2area(\Omega)}{\int_0^L \theta'(s)^p\,ds}(\theta'(s)^{p})'=(\gamma(s)-\bar\gamma,{\bf n})',\qquad \text{in }(0,L).
\end{equation}
Integrating the equation we have
\begin{equation}
\frac{2area(\Omega)}{\int_0^L \theta'(s)^p\,ds}\kappa^p=(\gamma(s)-\bar\gamma,{\bf n})+M,
\end{equation}
where $M$ is a constant. On the other hand, recalling that
\[\int_0^Lh(s)\,ds=2area(\Omega),
\]
we get $M=0$, that is 
\begin{equation}\label{andrews}
\kappa^p=\alpha(\gamma(s)-\bar\gamma)\cdot{\bf n},
\end{equation}
where $\alpha$ is a positive constant.
We can now apply Theorem 1.5 in \cite{A}, and we owe this idea to Remark 4.1 in \cite{BH}, to conclude that $\Omega$ must be a disc.
\end{proof}

\section{Proof of Theorem \ref{main} and Corollary \ref{corollary}}\label{sec_theorem}

\begin{proof}[Proof of  Theorem \ref{main}]

The basic strategy is a reduction to the convex case. If $\Omega$ is a simply connected smooth set with boundary $\partial\Omega$ of class $C^1$ and piecewise of class $C^2$, then $\kappa$ is well defined in all but finite number of points. To illustrate the idea we consider an easy case where the domain $\Omega$ is like the one in Figure \ref{reduction}. Assume we can find two chords that cut two disjoint convex sets, entirely contained in $\Omega$ (shaded in grey in Figure \ref{reduction}), such that the tangents to $\partial\Omega$ at the end points of each chord are parallel. By rotating the two convex sets around the midpoint of each chord we can construct two convex sets, namely $\Omega_1$ and $\Omega_2$, with boundaries of class $C^1$ and piecewise of class $C^2$.
Then we have
\[
E_f(\partial\Omega)\ge \frac{1}{2}\left(E_f(\partial\Omega_1)+E_f(\partial\Omega_2)\right) \ge  \frac{1}{2}\left(E_f(\partial D_1)+E_f(\partial D_2)\right).
\]
here $D_1$ and $D_1$ are discs of same area as $\Omega_1$ and $\Omega_2$ respectively.
If $D$ is the disc of area equal to $\frac12 \left(area(D_1)+area(D_2)\right)$, and if we denote by $R, R_1, R_2$ the radii of $D, D_1, D_2$, using assumption (2) we have
\[
\frac{1}{2}\left(E_f(\partial D_1)+E_f(\partial D_2)\right)=\pi\left(R_1f\left(\frac1{R_1}\right) + R_2f\left(\frac1{R_2}\right)\right)\ge 2\pi \left(Rf\left(\frac1{R}\right)\right)=E_f(\partial D).
\] 
Hence we are able to prove that $E_f(\partial\Omega)\ge E_f(\partial D)$ with $area(E)>area(D)$. To conclude the proof we observe that assumptions (1) and (2) imply 
\[
sf'(s)\ge\int_0^s r^{-1}f(r)dr \qquad\mbox{ for $s>0$,}
\]
and therefore $f$ is an increasing functions in $[0,\infty]$. As a consequence the energy $E_f$ of the boundary of a disc is a monotone decreasing function of the radius.

Obviously the previous proof relies on the construction of the two sets $\Omega_1$ and $\Omega_2$ which is not always trivial. In \cite{FKN} a set $\Omega$ like the one in Figure \ref{reduction} was labeled as a set \emph{holding two disjoint convex sets.}
Even if in general a set does not \emph{hold two disjoint convex sets}, the proof of \cite[proof of Theorem 1]{FKN} shows that it is always possible to cut and paste the boundary of a $C^2$ piecewise $C^1$ set, possibly gluing a finite number of segments, to construct a $C^2$ piecewise $C^1$ curve enclosing a domain $\widehat \Omega$ having smaller area and eventually \emph{holding two disjoint convex sets}. Since the curve $\gamma$ is constructed using only pieces of the boundary of $\partial \Omega$ together with a finite number of straight segments, it is trivial that $E_f(\partial\widehat\Omega)\le E_f(\partial\Omega)$. We observe that for non convex sets we establish a comparison with a disc of smaller area and therefore the equality case follows easily.

 
\begin{figure}
\def\svgwidth{6cm}
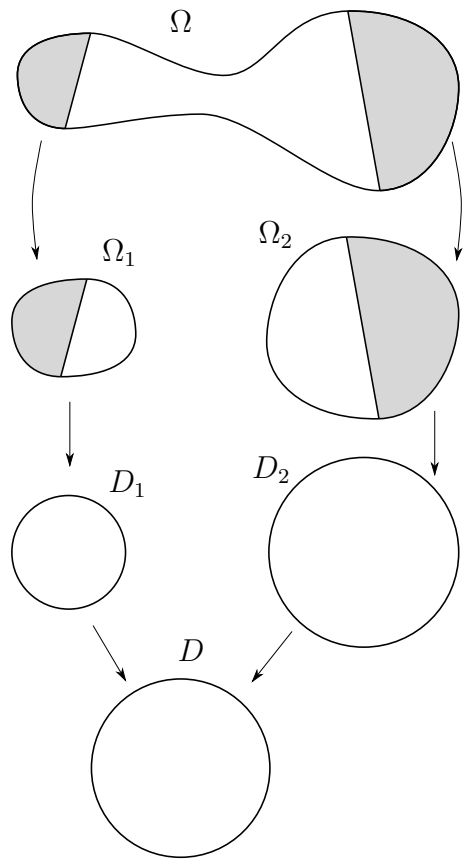
\caption{Comparing non convex sets to balls.}\label{reduction}
\end{figure}



\end{proof}

\begin{proof}[Proof of Corollary \ref{corollary}]
Let $1\le p\le\infty$. We can employ Theorem \ref{main} and Theorem \ref{prop} to establish \eqref{isop} for $C^1$ piecewise $C^2$ simple closed curves. The disc and only the disc achieves the equalities. For any $W^{2,p}$ simple closed curve $\gamma$ there exists a sequence of $C^1$ piecewise $C^2$ simple closed curves $\{\gamma_n\}_{n\in\mathbb N}$ so that
\[
\lim_n \|\gamma-\gamma_n\|_{W^{2,p}(0,L)}=0.
\]
Therefore \eqref{isop} are valid also for $W^{2,p}$ curves. By standard arguments in calculus of variation a $W^{2,p}$ simple closed curve achieving equality in \eqref{isop} is also analytic and therefore equalities can hold only for discs.

For $p=\infty$ we can pass to the limit $p\to\infty$ in \eqref{isop} as $p\to\infty$. The result was already known \cite{FNT}, \cite[Proposition 2.1]{HT}, \cite[Lemma 2.2]{Ka}, \cite{Pa}, \cite{PI},  where the equality case is also discussed.
\end{proof}

\end{document}